\documentclass[12pt,english,reqno]{amsart}  

\usepackage[T1]{fontenc} 
\usepackage{geometry} 
\usepackage{amssymb,amsmath,amsthm} 

\usepackage{pdfpages}
\usepackage{graphicx}
\usepackage{tikz} 
\usetikzlibrary{arrows,automata} 

\usepackage{hyperref}
\hypersetup{
    colorlinks=true,
    linkcolor=blue,
    filecolor=magenta,      
    urlcolor=cyan,
}

\usepackage[utf8]{inputenc}

\newtheorem{theorem}{Theorem}[section]
\newtheorem{corollary}[theorem]{Corollary}
\newtheorem{lemma}[theorem]{Lemma}

\newtheorem{definition}[theorem]{Definition}

\newcommand{\bfS}[1]{\mathbb{S}^{#1}} 
\newcommand{\ind}{\operatorname{ind}} 

\newcommand{\mbS}{\mathbb S}
\newcommand{\R}{\mathbb R}

\title[Unexpected multiplicity stability operator]{New Explicit Eigenfunctions of the stability operator on some minimal hypersurfaces}
\author{Oscar Perdomo}
\address{Central Connecticut State University}
\email{perdomoosm@ccsu.edu}
\date{\today}

\begin{document}

\maketitle

\begin{abstract}
For any $n$-dimensional compact minimal hypersurface of the $(n+1)$-dimensional sphere, we have that the coordinate functions of the Gauss map are eigenfunctions of the stability operator associated with the eigenvalue $-n$. In this paper we consider minimal  immersions from $\mbS^{k}\times\mbS^{\ell}\times\mbS^{1}$ to $\mbS^{k+\ell+2}$ of the form $\phi(y,z,t)=\left(f(t) y, f_2(t) z, f_1(t)\right)$  and we explicitly show new eigenfunctions for the stability operator associated with the same eigenvalue. We also show that the stability index of these minimal immersions is at least $k\ell+3k+3\ell+8$.
\end{abstract}

\section{Introduction}

Let $M\subset \mbS^{n+1}$ be a generalized hypersurface as defined in \cite{P1}. That is, for some hypersurface $M_0$ contained in the interior of the unit ball of $\R^{n-k+1}$, 

$$M=\{\left(\sqrt{1-|m|^2}\, y, m\right) : y\in\mbS^k, \, m\in M_0\}.$$

In the particular case when $M_0$ is a hypersurface  of revolution of the form $\left(f_2(t) z,f_1(t)\right)$, with  $\ell=n-k-1$, $z\in \mbS^{\ell}$ and $f_1(t)$ and $f_2(t)>0$ are periodic functions, the immersion $M$ can be written as

\begin{eqnarray}\label{ti}
\phi(y,z,t)=\left(f(t) y, f_2(t) z,f_1(t)\right)\quad\hbox{with}\quad f(t)=\sqrt{1-f_1(t)^2-f_2(t)^2}.
\end{eqnarray}

Carlotto and Schulz, \cite{CS}, showed the existence of embedded minimal hypersurfaces of the form \eqref{ti} for the case $k=\ell$. This collection of examples was extended in two different ways: (i) keeping the condition $k=\ell$ and changing the mean curvature condition $H=0$ to $H\ne 0$ \cite{HW,LW}, and (ii) keeping the condition $H=0$ and allowing $k$ and $\ell$ to be any pair of positive integers, \cite{FT, LWa}.

For the case $k=\ell=1$ the author showed the existence of a  CMC example  \cite{P4} by using the Poincare-Miranda theorem and a numerical method that keeps track of the round-off error \cite{P7}. The author has also shown numerical evidence that the list of embedded CMC examples that has been mathematically shown to exist is far from being complete, \cite{P5, P6}.


For a minimal hypersurface $M\subset \mbS^{n+1}$, the stability index of $M$, denoted by $\ind(M)$, is defined as the number of negative eigenvalues, counted with multiplicity, of the stability operator $J(f)=-\Delta f-nf-|A|^2f$, where $|A|^2=\kappa_1^2+\dots +\kappa_n^2$ and $\kappa_i$ are the principal curvatures of $M$. If the Gauss map $\nu:M\longrightarrow \mbS^{n+1}$ is  given by $\nu=\left( \nu_1,\dots, \nu_{n+2}\right)$, then the function $\nu_i$ satisfy that $J(\nu_i)=-n\nu_i$.  It is easy to check that when $M$ is not totally geodesic, the functions $\nu_i$ are linearly independent and therefore $-n$ is an eigenvalue of the stability operator with multiplicity at least $n+2$. Since the first eigenvalue of $J$ must have multiplicity $1$ we conclude that $\ind(M)\ge n+3$. The Clifford minimal hypersurfaces satisfy that $\ind(M)=n+3$ and a long-standing conjecture states that if $\ind(M)=n+3$ then $M$ must be a Clifford hypersurface. This conjecture has been verified only  in the case $n=2$ by Urbano \cite{U}. This result of Urbano was used in the proof of Willmore's Conjecture \cite{MN}.


With respect to the stability index of the Carlotto-Schulz minimal hypersurfaces, for $k=\ell=1$, Carlotto, Schulz, and Wiygul showed that the stability index is at least $8$, \cite{CSW}. The author showed that this index is at least $15$, \cite{P2}, and it can be shown numerically that the index is $27$. For the general case $k=\ell>1$, the stability index is at least $k^2+6k+8$, \cite{P2}, and the author has conjectured that the index is $(k+2)(k^2+10k+12)/3$. This conjecture has been numerically verified for $k=2,\dots,100$, \cite{P3}.

For the case $k=\ell$, the author has shown that the multiplicity of $-(2k+1)$ as an eigenvalue of the stability operator is at least $(2k+3)+(k+1)^2$, \cite{P3}. In this paper, we generalize this result to the case $k\ne \ell$ and show that the multiplicity of $-(k+\ell+1)$ as an eigenvalue of the stability operator is at least $(k+\ell+3)+(k+1)(\ell+1)$.

\section{New eigenfunctions contributing to the stability index}

We have the following expression for the Laplacian of functions on the immersion \eqref{ti}. Its proof can be found in \cite{P1}.

\begin{lemma}\label{formula-laplacian}
Let $\zeta(t,y,z)$ be a smooth function defined along the immersion

$$
\phi(y,z,t)=\left(f(t)y,f_2(t)z,f_1(t)\right)
\quad\hbox{with}\quad
f(t)=\sqrt{1-f_1(t)^2-f_2(t)^2}.
$$

Then
\[
\Delta\zeta
=\frac{1}{1+(f^\prime)^2}\left( \zeta_{tt}
+b_{k\ell} \zeta_t\right)
+\frac{1}{f^2}\Delta_{\bfS{k}}^y \zeta
+\frac{1}{f_2^2}\Delta_{\bfS{\ell}}^z \zeta,
\]
where
\[
b_{k\ell}=k\frac{f'}{f}+\ell \frac{f_2'}{f_2}-\frac{f'f''}{1+(f')^2},
\]
and $\Delta_{\bfS{k}}^y\zeta$ (respectively, $\Delta_{\bfS{\ell}}^z\zeta$) denotes the Laplacian on the sphere $\bfS{k}$ acting on the $y$-variable (respectively, on $\bfS{\ell}$ acting on the $z$-variable), with $t$ and $z$ (respectively, $t$ and $y$) held fixed.
\end{lemma}


We now prove the main theorem in this paper.

\begin{theorem}\label{thm1}  Let \(f_1\) and \(f_2\) be the functions that define a minimal immersion of the form \eqref{ti}. For any $i\in\{1,\dots, k+1\}$ and $j\in\{1,\dots, \ell+1\}$, the functions   

$$\zeta_{ij}=\omega(t)y_iz_j \quad\hbox{with}\quad \omega = f^{-k}f_2^{-\ell}$$ 

satisfy that

$$J(\zeta_{ij})=-n\zeta_{ij}$$
\end{theorem}

\begin{proof}

Since $\Delta_{\bfS{k}}^y y_i=-ky_i$ and $\Delta_{\bfS{\ell}}^z z_j=-\ell z_j$,  from Lemma \ref{formula-laplacian}, we need to show that 

\begin{eqnarray}\label{maineqn}
\frac{1}{1+(f^\prime)^2}\left( \omega''
+b_{k\ell} \omega'\right)+|A|^2\omega
-\frac{k}{f^2}\omega
-\frac{\ell}{f_2^2}\omega=0,
\end{eqnarray}

The rest of this proof will be devoted to showing the details of this identity. 

We assume that the curve $(f_1,f_2)$ (which will be called the profile curve) is parametrized by arc-length and 
we will be using the so-called TreadmillSled coordinates introduced by the author in the
study of helicoidal surfaces (see for example \cite{P9, P10}). The TreadmillSled coordinates $(\xi_1,\xi_2)$ of the curve $(f_1,f_2)$ are defined by

\begin{eqnarray*}
\xi_1=f_1\cos(\theta)+f_2\sin(\theta)\quad\hbox{and}\quad \xi_2=f_2\cos(\theta)-f_1\sin(\theta)
\end{eqnarray*}

From \cite{P1} we know that the minimality of the immersion $M$ is equivalent to the following system of differential equations:
\[
f_1'=\cos\theta,\qquad
f_2'=\sin\theta,\qquad
\theta'=\frac{h^2}{f_2f^2}\bigl(\ell\cos\theta-nf_2\xi_2\bigr),
\]

where

\begin{equation}\label{defh}
h=\sqrt{1-\xi_2^2}
\end{equation}

The following identities can be verified directly.

\begin{eqnarray}
h^2&=&f^2+\xi_1^2\\
 f'&=&-\frac{\xi_1}{f}\\
  \xi_1'&=&1+\xi_2\theta'\\
  \xi_2'&=&-\xi_1\theta'\\
  f''&=&-\frac{h^2+f^2\xi_2\theta'}{f^3}\\
  a_1&=&\frac{1}{1+(f')^2}=\frac{f^2}{h^2}
\end{eqnarray}

We point out that there is a typo in the last formula in \cite{P1}, where $a_1$ appears as $\frac{h^2}{f^2}$ instead of $\frac{f^2}{h^2}$. There are three principal curvatures: $\kappa_y$ with multiplicity $k$ and principal directions tangent to the $\mbS^k$ factor, $\kappa_z$ with multiplicity $\ell$ and principal directions tangent to the $\mbS^{\ell}$ factor, and $\kappa_t$ with multiplicity one and principal direction $\frac{\partial \phi}{\partial t}$.

Using the formulas obtained in  \cite{P1} and the identities above we obtain that 
\begin{eqnarray}
\kappa_y&=&\frac{\xi_2}{h} \\
\kappa_z&=&\frac{\xi_2f_2-\cos\theta}{hf_2} \\
\kappa_t&=&\frac{\xi_2+\theta'}{h}-\frac{(ff')^2\theta'}{h^3 }=\frac{f^2\theta'+\xi_2 h^2}{h^3}
\end{eqnarray}

Therefore, $|A|^2=k\kappa_y^2+\ell \kappa_z^2+\kappa_t^2$ reduces to,

$$|A|^2=n\frac{\xi_2^2}{h^2}+\ell \frac{\cos^2\theta}{h^2f_2^2}-2\ell \cos\theta \frac{\xi_2}{h^2f_2}+\frac{f^2\theta'}{h^4}\left(\ell \frac{\cos\theta}{f_2}-(n-2)\xi_2\right) $$

If we define

$$u(t)=\log \omega(t)=-k\log f(t)-\ell\log f_2(t),$$

then we have that

\[
\omega'=u'\omega,
\qquad
\omega''=(u''+(u')^2)\omega.
\]

A direct verification shows that the expression on the left-hand side of Equation \eqref{maineqn}, divided by \(\omega\), reduces to

$$Q=
\frac{1}{1+(f')^2}\left(u''+(u')^2+b_{k\ell}u'\right)
+|A|^2-\frac{k}{f^2}-\frac{\ell}{f_2^2}.
$$

We have that $u'=k\frac{\xi_1}{f^2}-\ell\frac{\sin\theta}{f_2}$ and 

\[
u''=k\left(\frac{2\xi_1^2}{f^4}
+\frac{1}{f^2}
+\frac{\xi_2}{f^2}\theta'\right)
+\ell \frac{\sin^2\theta}{f_2^2}
-\ell\frac{\theta'}{f_2}\cos\theta.
\]

Since $b_{k\ell}=-u'-a_1 f'f''$, the quantity $Q$, which we want to show vanishes, reduces to
\[
Q=
\frac{f^2}{h^2}
\left(
u''-\frac{f^2}{h^2}u'f'f''
\right)
+|A|^2-\frac{k}{f^2}-\frac{\ell}{f_2^2}.
\]

Using the identities for \(u'\), \(f'\), and \(f''\), we obtain

$$\frac{f^2}{h^2}u'f'f''=k\left(\frac{\xi_1^2}{f^4}+\frac{\xi_2\xi_1^2}{f^2h^2}\theta'\right)-\frac{\ell \xi_1\sin\theta}{f^2f_2}-\frac{\ell\xi_1\xi_2\sin\theta}{h^2f_2}\theta'$$

Using the formula for \(u''\) and the identities above, we obtain
\[
u''-\frac{f^2}{h^2}u'f'f''
=
k\frac{h^2}{f^4}
+\ell\frac{\sin^2\theta}{f_2^2}
+\ell\frac{\xi_1\sin\theta}{f_2f^2}
+
\left(
\ell\frac{\xi_1\xi_2\sin\theta}{f_2h^2}
+k\frac{\xi_2}{h^2}
-\ell\frac{\cos\theta}{f_2}
\right)\theta'.
\]
Multiplying by \(\frac{f^2}{h^2}\), we get
\[
\frac{f^2}{h^2}
\left(
u''-\frac{f^2}{h^2}u'f'f''
\right)
=
\frac{k}{f^2}
+\ell\frac{f^2\sin^2\theta}{h^2f_2^2}
+\ell\frac{\xi_1\sin\theta}{f_2h^2}
+
\left(
\ell\frac{f^2\xi_1\xi_2\sin\theta}{f_2h^4}
+k\frac{f^2\xi_2}{h^4}
-\ell\frac{f^2\cos\theta}{h^2f_2}
\right)\theta'.
\]

Using the formula for \(|A|^2\), we have
\[
\begin{aligned}
Q
={}&
\frac{k}{f^2}
+\ell\frac{f^2\sin^2\theta}{h^2f_2^2}
+\ell\frac{\xi_1\sin\theta}{f_2h^2} \\
&+
\left(
\ell\frac{f^2\xi_1\xi_2\sin\theta}{f_2h^4}
+k\frac{f^2\xi_2}{h^4}
-\ell\frac{f^2\cos\theta}{h^2f_2}
\right)\theta' \\
&+
n\frac{\xi_2^2}{h^2}
+\ell\frac{\cos^2\theta}{h^2f_2^2}
-2\ell\frac{\xi_2\cos\theta}{h^2f_2} \\
&+
\frac{f^2\theta'}{h^4}
\left(
\ell\frac{\cos\theta}{f_2}
-(n-2)\xi_2
\right)
-\frac{k}{f^2}
-\frac{\ell}{f_2^2}.
\end{aligned}
\]
The terms \(\frac{k}{f^2}\) and \(-\frac{k}{f^2}\) cancel. Hence
\[
\begin{aligned}
Q
={}&
\ell\frac{f^2\sin^2\theta}{h^2f_2^2}
+\ell\frac{\xi_1\sin\theta}{f_2h^2}
+n\frac{\xi_2^2}{h^2}
+\ell\frac{\cos^2\theta}{h^2f_2^2}
-2\ell\frac{\xi_2\cos\theta}{h^2f_2}
-\frac{\ell}{f_2^2} \\
&+
\left(
\ell\frac{f^2\xi_1\xi_2\sin\theta}{f_2h^4}
+k\frac{f^2\xi_2}{h^4}
-\ell\frac{f^2\cos\theta}{h^2f_2}
+\ell\frac{f^2\cos\theta}{h^4f_2}
-(n-2)\frac{f^2\xi_2}{h^4}
\right)\theta'.
\end{aligned}
\]


We now simplify the coefficient of \(\theta'\). Since \(n=k+\ell+1\), we have
\[
k-(n-2)=1-\ell.
\]
Therefore
\[
\begin{aligned}
&\ell\frac{f^2\xi_1\xi_2\sin\theta}{f_2h^4}
+k\frac{f^2\xi_2}{h^4}
-\ell\frac{f^2\cos\theta}{h^2f_2}
+\ell\frac{f^2\cos\theta}{h^4f_2}
-(n-2)\frac{f^2\xi_2}{h^4} \\
&=
\frac{f^2}{h^4}
\left(
\ell\frac{\xi_1\xi_2\sin\theta}{f_2}
+(1-\ell)\xi_2
-\ell\frac{h^2\cos\theta}{f_2}
+\ell\frac{\cos\theta}{f_2}
\right).
\end{aligned}
\]


Using \(h^2=1-\xi_2^2\), we get
\[
-\ell\frac{h^2\cos\theta}{f_2}
+\ell\frac{\cos\theta}{f_2}
=
\ell\frac{\xi_2^2\cos\theta}{f_2}.
\]
Hence the coefficient of \(\theta'\) in the last expression for $Q$ is
\[
\frac{f^2}{h^4}
\left(
\ell\frac{\xi_1\xi_2\sin\theta}{f_2}
+\ell\frac{\xi_2^2\cos\theta}{f_2}
+(1-\ell)\xi_2
\right).
\]
Factoring \(\xi_2\), this becomes
\[
\frac{f^2\xi_2}{h^4}
\left(
\ell\frac{\xi_1\sin\theta+\xi_2\cos\theta}{f_2}
+1-\ell
\right).
\]
Since
\[
f_2=\xi_2\cos\theta+\xi_1\sin\theta,
\]
we obtain
\[
\ell\frac{\xi_1\sin\theta+\xi_2\cos\theta}{f_2}
+1-\ell
=
\ell+1-\ell=1.
\]
Therefore the coefficient of \(\theta'\) is
\[
\frac{f^2\xi_2}{h^4}.
\]


It follows that
\[
Q=
\ell\frac{f^2\sin^2\theta}{h^2f_2^2}
+\ell\frac{\xi_1\sin\theta}{f_2h^2}
+n\frac{\xi_2^2}{h^2}
+\ell\frac{\cos^2\theta}{h^2f_2^2}
-2\ell\frac{\xi_2\cos\theta}{h^2f_2}
-\frac{\ell}{f_2^2}
+
\frac{f^2\xi_2}{h^4}\theta'.
\]
Using the differential equation
\[
\theta'
=
\frac{h^2}{f_2f^2}
\left(
\ell\cos\theta-nf_2\xi_2
\right),
\]
we obtain
\[
\frac{f^2\xi_2}{h^4}\theta'
=
\frac{\xi_2}{f_2h^2}
\left(
\ell\cos\theta-nf_2\xi_2
\right).
\]
Therefore
\[
Q=
\ell\frac{f^2\sin^2\theta}{h^2f_2^2}
+\ell\frac{\xi_1\sin\theta}{f_2h^2}
+n\frac{\xi_2^2}{h^2}
+\ell\frac{\cos^2\theta}{h^2f_2^2}
-2\ell\frac{\xi_2\cos\theta}{h^2f_2}
-\frac{\ell}{f_2^2}
+
\frac{\xi_2}{f_2h^2}
\left(
\ell\cos\theta-nf_2\xi_2
\right).
\]


Multiplying by \(h^2f_2^2\), it is enough to verify that

\[
\ell f^2\sin^2\theta
+\ell\xi_1f_2\sin\theta
+n\xi_2^2f_2^2
+\ell\cos^2\theta
-2\ell\xi_2f_2\cos\theta
-\ell h^2
+\xi_2f_2
\left(
\ell\cos\theta-nf_2\xi_2
\right)
=0.
\]
The terms involving \(n\) cancel, and after factoring \(\ell\), it remains to show that
\[
f^2\sin^2\theta
+\xi_1f_2\sin\theta
+\cos^2\theta
-\xi_2f_2\cos\theta
-h^2
=0.
\]


Using
\[
f^2=1-f_1^2-f_2^2,
\qquad
h^2=1-\xi_2^2,
\]
the left-hand side becomes
\[
-(f_1^2+f_2^2)\sin^2\theta
+\xi_1f_2\sin\theta
-\xi_2f_2\cos\theta
+\xi_2^2.
\]
Finally, substituting
\[
\xi_1=f_1\cos\theta+f_2\sin\theta,
\qquad
\xi_2=f_2\cos\theta-f_1\sin\theta,
\]
we obtain
\[
\begin{aligned}
&-(f_1^2+f_2^2)\sin^2\theta
+f_2\sin\theta
\left(
f_1\cos\theta+f_2\sin\theta
\right) \\
&\quad
-f_2\cos\theta
\left(
f_2\cos\theta-f_1\sin\theta
\right)
+
\left(
f_2\cos\theta-f_1\sin\theta
\right)^2.
\end{aligned}
\]
Expanding the last expression gives
\[
\begin{aligned}
&-f_1^2\sin^2\theta
-f_2^2\sin^2\theta
+f_1f_2\sin\theta\cos\theta
+f_2^2\sin^2\theta \\
&\quad
-f_2^2\cos^2\theta
+f_1f_2\sin\theta\cos\theta
+f_2^2\cos^2\theta
-2f_1f_2\sin\theta\cos\theta
+f_1^2\sin^2\theta \\
&=0.
\end{aligned}
\]
Therefore \(Q=0\), and Equation \eqref{maineqn} follows.

\end{proof}

For the Carlotto--Schulz minimal examples, the author has shown that their stability index is at least \(k^2+6k+8\); recall that in this case \(k=\ell\). It is unknown if there are more embedded minimal generalized rotational hypersurfaces with \(k=\ell\). In general, for any \(k\) and \(\ell\), regardless of whether the immersion is embedded or not, the stability index of the immersion \eqref{ti} is at least \(k\ell+3k+3\ell+8\). Let us see this.

\begin{corollary}
If the functions \(f_1\) and \(f_2>0\) are \(T\)-periodic, \(f=\sqrt{1-f_1^2-f_2^2}>0\), and the immersion \eqref{ti} is minimal, then the stability index of this immersion is at least \(k\ell+3k+3\ell+8\).
\end{corollary}

\begin{proof}
Following the ideas in \cite{P1,P2}, each negative eigenvalue of the operator
\[
S_{11}(\eta)
=
-\frac{1}{1+(f^\prime)^2}
\left(
\eta''+b_{k\ell}\eta'
\right)
-|A|^2\eta-n\eta
\]
contributes to the stability index. Moreover, \(-n\) is an eigenvalue of \(S_{11}\). This is an expected eigenvalue coming from the coordinates of the Gauss map.

We also have that each negative eigenvalue of the operator
\[
S_{21}(\eta)
=
-\frac{1}{1+(f^\prime)^2}
\left(
\eta''+b_{k\ell}\eta'
\right)
-|A|^2\eta-n\eta
+\frac{k}{f^2}\eta
\]
contributes \(k+1\) toward the stability index, and we know that \(-n\) is an expected eigenvalue.

Likewise, each negative eigenvalue of the operator
\[
S_{12}(\eta)
=
-\frac{1}{1+(f^\prime)^2}
\left(
\eta''+b_{k\ell}\eta'
\right)
-|A|^2\eta-n\eta
+\frac{\ell}{f_2^2}\eta
\]
contributes \(\ell+1\) toward the stability index, and we know that \(-n\) is an expected eigenvalue.

Finally, each negative eigenvalue of the operator
\[
S_{22}(\eta)
=
-\frac{1}{1+(f^\prime)^2}
\left(
\eta''+b_{k\ell}\eta'
\right)
-|A|^2\eta-n\eta
+\frac{k}{f^2}\eta
+\frac{\ell}{f_2^2}\eta
\]
contributes \((k+1)(\ell+1)\) toward the stability index. By Theorem \ref{thm1}, \(-n\) is an unexpected eigenvalue of \(S_{22}\).

By comparing the operators \(S_{21}\) and \(S_{12}\) with \(S_{22}\), the Rayleigh characterization of eigenvalues implies that \(-n\) is at least the second eigenvalue of \(S_{21}\) and \(S_{12}\). Now comparing \(S_{11}\) with \(S_{12}\), we obtain that \(-n\) is at least the third eigenvalue of \(S_{11}\). It follows that the stability index is at least
\[
3+2(k+1)+2(\ell+1)+(k+1)(\ell+1)=k\ell+3k+3\ell+8.
\]
\end{proof}

\end{document}